\documentclass[10pt]{amsart}
\usepackage{amssymb,amsmath,amsthm}

\newcommand{\N}{{\ensuremath{\mathbb{N}}}}
\newcommand{\Z}{{\ensuremath{\mathbb{Z}}}}

\newcommand{\R}{{\ensuremath{\mathbb{R}}}}
\newcommand{\C}{{\ensuremath{\mathbb{C}}}}

\newcommand{\norm}[1]{\left\Vert#1\right\Vert}
\newcommand{\e}[1]{\ensuremath{e^{2\pi i#1}}}
\newcommand{\ee}[1]{\ensuremath{e^{-2\pi i#1}}}
\newcommand{\ip}[2]{\ensuremath{\left<#1,#2\right>}}
\newcommand{\set}[1]{\ensuremath{\left \{ #1 \right \}}}
\newcommand{\abs}[1]{\ensuremath{\left| #1 \right| }}

\newcommand{\ceil}[1]{\ensuremath{\lceil #1 \rceil}}

\newcommand{\osc}{\operatorname{osc}}

\newcommand{\ltwo}{\ell^2(\Z^d)}

\newcommand{\momspace}[1]{\mathcal{M}^{#1}}
\newcommand{\momd}[1]{M^{#1}}
\newcommand{\mom}[2]{M^{#1}_{#2}}

\newcommand{\f}{\mathcal{F}}

\newcommand{\deriv}[1]{\partial_{#1}}
\newcommand{\de}[1]{D^{#1}}
\newcommand{\sob}[1]{W^{#1}_\infty}

\newtheorem{theo}{Theorem}[section]
\newtheorem{prop}{Proposition}[section]
\newtheorem{coro}[theo]{Corollary}
\newtheorem{lemma}[theo]{Lemma}

\newtheorem{fact}{Fact}[section]

\newtheorem{rem}{Remark}[section]

\newtheorem*{theo*}{Theorem}
\newtheorem*{prop*}{Proposition}
\newtheorem*{coro*}{Corollary}
\newtheorem*{lemma*}{Lemma}
\newtheorem*{claim*}{Claim}
\newtheorem*{fact*}{Fact}
\newtheorem*{deff*}{Definition}
\newtheorem*{obs*}{Observation}
\newtheorem*{names*}{Names}

\begin{document}

\title{A deconvolution estimate and localization in spline-type spaces}

\author[Jos\'e Luis Romero]{Jos\'e Luis Romero}
\address{Departamento de Matem\'atica \\ Facultad de Ciencias Exactas y Naturales\\ Universidad
de Buenos Aires \\ and CONICET, Argentina}
\email[Jos\'e Luis Romero]{jlromero@dm.uba.ar}

\date{\today}

\begin{abstract}
In this article some explicit estimates on the decay of the convolutive inverse of
a sequence are proved. They are derived from the functional calculus for
Sobolev algebras. Applications include localization in spline-type spaces and
oversampling schemes.
\end{abstract}
\maketitle

\section{Introduction}
The deconvolution problem is the one of solving a convolution equation of the form
$a*x=y$, where $a$ is a fixed sequence, $y$ is a given signal of some type and $x$ is unknown.
This problem arises naturally on a wide number of fields. The formal general solution to the problem
is $x = b*y$, where $b$ is the convolutive inverse of $a$, that is $a*b=\delta$.

In applications, the equation $a*x=y$ can be solved by first computing some approximation
of the convolutive inverse element $b$ and then for every data $y$ computing the convolution $b*y$, or
by approximately solving for $x$ by means of some iterative algorithm. In order to evaluate
the performance of such schemes it is convenient to have some bounds on the decay of the sequence $b$.

There exist general decay principles asserting that if the original sequence $a$ has some good decay properties, and $b$ is only assumed to have some mild decay, then it is the case that $b$ has also good decay.
These principles often assert the preservation under convolutive inversion of the finiteness of
some quantity that measures decay. Consequently, in these cases, from specific decay information on $a$
only asymptotic decay information on $b$ is obtained. Moreover, for some of these decay principles
it is known that bounds on the quantity used to measure the decay of $a$, although enough to ensure
the finiteness of the same decay measure for $b$, are in fact insufficient to give any bound on it.
This means that even when using asymptotic decay estimates on the inverse convolutive element $b$ only
for some quality analysis purpose, the estimates thus obtained do not fully express all the qualities of $a$
involved in such an estimate.

The purpose of this article is to point out that the well-known functional calculus for Sobolev algebras
can be used to give explicit and general estimates on the decay of $b$ from corresponding estimates
on the decay of $a$. The estimates obtained in this article are sharp for very concentrated
sequences $a$ and tend to be coarser for not so concentrated sequences. They are not intended
to replace any numerical or case-specific bounds but to give a general a priori estimate that may be needed
to rigorously use some sharper techniques.

In section \ref{results} we introduce the quantities used to measure decay and present the announced
estimates. The next section provides the proofs. In section \ref{apps} we present some applications
to the localization of spline-type spaces and irregular sampling.

For a one dimensional sequence of complex numbers $a \equiv \set{a_k}_{k \in \Z} \in \ell^2$
satisfying the decay condition $\sum_k k^2 \abs{a_k}^2 < \infty$ we give the following 
estimate: if $a$ has a convolutive inverse $b \in \ell^2$, then $b$ satisfies the same
decay condition: $\sum_k k^2 \abs{b_k}^2 < \infty$. Moreover, we give a concrete bound on that
number (see Theorem \ref{one_theo}.)
This is somehow in the spirit of Wiener's classical $1/f$ lemma that asserts that if a
sequence $a \in \ell^1$ has a convolutive inverse $b \in \ell^2$, then $b$ must belong to $\ell^1$.

Observe that the quantity preserved by our estimate is the $\ell^2$ norm of the weighted sequence
$\set{ka_k}_{k\in\Z}$. For more dimensions and stronger decay conditions our estimate are more
complicated. We consider a sequence $a \equiv \set{a_k}_{k \in \Z^d} \in \ell^2$ and measure
its decay by imposing restrictions on the weighted sequence $\set{k^\alpha a_k}_{k \in \Z^d}$
where $\alpha$ is a multi-index (see definitions in Section \ref{prelims}.) The quantity preserved
from $a$ to $b$ is more difficult to interpret in this case; it is the $L^\infty$ norm of the Fourier
transform of the weighted sequence $\set{k^\alpha a_k}_{k \in \Z^d} \in \ell^2$. Nevertheless, this quantity
can be bounded above and below by the $\ell^1$ and $\ell^2$ norms of that sequence respectively, which
are quantities that have a clear meaning. Consequently, from the practical point of view, the estimates
on the decay of $b$ follow from slightly stronger estimates on the decay of $a$. Yet, the fact remains
that any degree of decay of $b$ can be granted by requiring a sufficiently high degree of decay of $a$.

For dimension greater that 1 and heavy weights the estimates we give are recursive: bounds
for the weighted version of $b$, $\set{k^\alpha b_k}_k$ are given
in terms of the weighted versions of $a$ and $b$, $\set{k^\beta a_k}_k$ and
$\set{k^\gamma b_k}_k$ with $\beta \leq \alpha$ and $\gamma < \alpha$ (where
inequalities are to be interpreted coordinatewise.) So, concrete estimates are obtained by
iterating these estimates. Observe that bounds on the weights $k^\alpha$ applied to $b$
are derived from bounds on the weights $k^\beta$ appplies to $a$ for all $\beta \leq \alpha$.
Hence, our approach allows for the preservation of anisotropic weights (that is, weights that
give more importance to certain directions). For example in dimension $2$ suppose that a sequence
$\set{a_{k,j}}_{k,j \in \Z^d}$ satisfies a decay condition of the form
$\sum_{k,j} (\abs{k}^5 + \abs{j}^2) \abs{a_{k,j}} < \infty$.
Then, if $a$ has a convolutive inverse $b \in \ell^2$, it will satisfy
$\sum_{k,j} (\abs{k}^5 + \abs{j}^2)^2 \abs{b_{k,j}}^2 < \infty$.
This can be seen by considering the weights $k^5=(k,j)^{(5,0)}$ and $j^2=(k,j)^{(0,2)}$
separately and adding the resulting estimates for $b$.

A (principal) spline-type space is a subspace of $L^p(\R^d)$ having a $p$-Riesz
basis\footnote{See definition in Section \ref{apps}.}
of integer translates of a generator function $\varphi$ satisfying some smoothness and decay requirements.
When $p=2$ it is well-known that the Riesz basis condition is equivalent to the boundedness
above and below of the so called \emph{gramian function}
$\sum_{k \in \Z^d} \abs{\hat{\varphi}(\cdot-k)}^2$ and is therefore much easier to grant than
it is for the general value of $p$. Moreover, in this case, the basis bounds are given by the essential infimum
and supremum of the gramian function. It has been shown (see \cite{algr01}) that deconvolution estimates
can be used to ensure that, for a well-localized generator function $\varphi$, if its integer translates
form an $2$-Riesz sequence, then they form a $p$-Riesz sequence for the whole range $1\leq p\leq\infty$.
This is sometimes referred to as a localization principle (see \cite{gr04} and \cite{bacahela06}.)
If this principle is based on an asymptotic deconvolution estimate it is not clear how the $p$-Riesz basis bounds
depend on the generator function $\varphi$. In section \ref{apps} we apply our deconvolution
bounds to the problem of localization in spline-type spaces to obtain explicit bounds
for the basis constants for the whole range of values of $p$ (Corollary \ref{app_loc_coro}.)
This yields some qualitative conclusions on these bounds. The results are given in dimension
1 for simplicity but the techniques they rely on are available for arbitrary dimensions.
We show that if the set of translates $\set{\varphi(\cdot-k): k \in \Z}$ of a well-localized
function $\varphi$ form a $L^2$-Riesz sequence with lower bound $A$, then they form
an $L^p$-Riesz sequence with lower bound $\approx \frac{A^2}{A+1}$, for all $1\leq p \leq\infty$.
Moreover, the implicit constants depends only and continuously on the parameters measuring the decay of $\varphi$. This can be viewed as a stability result and was the motivating question for this article.

As a final application we present some consequences for a well-know sampling scheme.
In \cite{alfe98}, Aldroubi and Feichtinger developed a general sampling method for spline-type
spaces and proved that any sufficiently dense set is adequate for reconstruction. Moreover
an effective reconstruction algorithm is given. The proof is based on the fact that functions
in a spline-type space have a controlled modulus of continuity. The technique of deriving
sampling inequalities from oscillation bounds had a lasting influence on the sampling literature.
As shown in \cite{alfe98}, in order to use their general method to derive explicit sufficient-density
bounds it is necessary to have explicit bounds for the $p$-Riesz sequence of translates
$\set{\varphi(\cdot-k): k \in \Z^d}$. We use the bounds derived in Section \ref{apps}
to give a version of the sampling theorem in \cite{alfe98}
where the prescribed sampling density and the sampling bounds are absolutely explicit,
though in general coarse.
The result is intended to shed some light on the question of what qualities of the sampling
model are involved in the prescribed sampling density and bounds. It should not be considered
as a replacement for specific sampling theorems concerning specific function classes.
\section{Preliminaries and notation}
\label{prelims}
Throughout the article $d$ will be a fixed integer (the dimension) and $I^d$
will denote the unitary cube $[0,1]^d$.
For a real number $x$, its \emph{ceiling} $\ceil{x}$ is the smallest integer
greater or equal to $x$.

We call an element of $\N_0^d$ a \emph{multi-index}. For a multi-index $\alpha$
we define its \emph{size} by $\abs{\alpha}=\sum_j \alpha_j$.
Given two multi-indexes $\alpha, \beta$ we 
say that $\beta \leq \alpha$ if $\beta_j \leq \alpha_j$ for all $j$
and set
\[
\binom{\alpha}{\beta} := \prod_{i=1}^d \binom{\alpha_i}{\beta_i}.
\]

For a sequence of complex numbers $a \equiv \set{a_k}_{k \in \Z^d}$
and a multi-index $\alpha \in \N_0^d$, let us set
\[
\momd{\alpha}(a)_k := k^\alpha a_k = k_1^{\alpha_1} \cdots k_d^{\alpha_d} a_k.
\]

For a locally integrable function $f:I^d \to \C$ and a multi-index $\alpha$
we denote by $\de{\alpha}(f)$ the distributional (weak) derivative of $f$,
$\de{\alpha}(f) = \partial_{x_1}^{\alpha_1} \ldots \partial_{x_d}^{\alpha_d}f$.

For a sequence $a \in \ell^2(\Z^d)$ we use the following version of
the discrete Fourier transform,
\[
\hat{a} := \sum_k a_k \e{k \cdot},
\]
and denote by $\f$ the operator $a \mapsto \hat{a}$. Then,
the following relation holds,
\[
\de{\alpha} \f = (2\pi i)^\alpha \momd{\alpha} \f.
\]
Similarly, for a function $f \in L^2(\R^d)$ we use the following version of
the continuous Fourier transform,
\[
\hat{f}(w):=\int_{\R^d} f(x) \ee{\ip{w}{x}} dx.
\]

If $E \subset L^2(I^d)$ is a subspace with its own norm $\norm{\cdot}_E$,
we set 
\[
\f^{-1}(E) := \set{a \in \ltwo: \hat{a} \in E}
\]
and endow it with the norm $\norm{a}_{f^{-1}(E)} := \norm{\hat{a}}_E$.
In the case $E=L^\infty(I^d)$ we will use the notation
$\norm{a}_{op} := \norm{\hat{a}}_{L^\infty(I^d)}$.

\section{Main results}
\label{results}
For a multi-index $\alpha \in \N_0^d$ and a Banach space of sequences $B$, let us define,
\[
\momspace{\alpha}(B) := \set{a \in B: \momd{\beta}(a) \in B \mbox{, for all $\beta \leq \alpha$}}.
\]
For $a \in \momspace{\alpha}(B)$, we call the following quantity
\emph{the $(\alpha, B)$ momentum} of $a$,
\[
\mom{\alpha}{B}(a) := \norm{\momd{\alpha}(a)}_B.
\]
When $B=\ell^p$, we shall use the notation $\momspace{\alpha}_p$ and 
$\mom{\alpha}{p}$, for $\momspace{\alpha}(\ell^p)$ and $\mom{\alpha}{\ell^p}$
respectively. We will also write $\momspace{\alpha}_{op}=\momspace{\alpha}(\f^{-1}(L^\infty))$
and $\mom{\alpha}{op} = \mom{\alpha}{\f^{-1}(L^\infty)}$.

The classes $\momspace{\alpha}(B)$ can be used to model decay. Different choices of the space $B$
and the multi-index $\alpha$ produce various decay requirements, privileging some directions over others.

We now claim that if a sequence in $\momspace{\alpha}(B)$ has a convolutive inverse
in $\ell^2$, then that inverse must belong to $\momspace{\alpha}(B)$. Moreover,
we give some estimates on the quantities involved.

\begin{theo}
\label{multi_theo}
Let $a \in \momspace{\alpha}_{op}(\Z^d)$, satisfy $0 \leq A \leq \abs{\hat{a}} \leq B < \infty$
for some constants $A,B$. Let $b \in \ltwo$ be defined by $\hat{b}:=1/\hat{a}$.
Then $b \in \momspace{\alpha}_{op}(\Z^d)$. Moreover, the following recursive estimates hold:
\begin{align*}
\mom{\gamma}{op}(b) &\leq
1/A \sum_{\beta < \gamma} \binom{\gamma}{\beta} \mom{\beta}{op}(b) \mom{\gamma-\beta}{op}(a)
\mbox{, for all $0 \not= \gamma \leq \alpha$,}
\\
\mom{0}{op}(b) &\leq 1/A.
\end{align*}
\end{theo}
\begin{rem}
Using different multi-indexes it is possible to put more weight
on some particular directions. The result shows that the privileged concentration
in those directions is preserved from $a$ to $b$.
\end{rem}
It will be seen (see Remark \ref{conv_alg}) that $\momspace{\alpha}_{op}$ is
a convolution algebra. This gives meaning to the following Corollary.
\begin{coro}
\label{multi_invariance}
Let $\alpha$ be multi-index. Then, the inclusion
\[
\momspace{\alpha}_{op}(\Z^d) \hookrightarrow B(\ltwo),
\quad a \mapsto a*-
\]
preserves the spectrum of each element.
\end{coro}

Theorem \ref{multi_theo} establishes the preservation under inversion
of the momenta $\mom{\alpha}{op}$. This quantities can result, however,
difficult to interpret. In order to derive concretes estimates the following
embeddings and inequalities can be used:
\begin{align*}
\momspace{\alpha}_1 \hookrightarrow \momspace{\alpha}_{op} \hookrightarrow \momspace{\alpha}_2,
\\
\mom{\alpha}{2} \leq \mom{\alpha}{op} \leq \mom{\alpha}{1}.
\end{align*}
The use of these relations together with the estimates in Theorem \ref{multi_theo}
result in an asymmetric corollary where bounds on the numbers $\mom{\gamma}{2}(b)$
are derived from bounds on $\mom{\gamma}{1}(a)$. For the one-dimensional case
we give a symmetric estimate where a bound on $\mom{1}{2}(b)$
is given in terms of the number $\mom{1}{2}(a)$.

\begin{theo}
\label{one_theo}
Let $a \in \momspace{1}_2(\Z)$ be such that $0 < A \leq \abs{\hat{a}} \leq B < \infty$,
for some constants $A,B$. Let $b \in \ell^2(\Z)$ be defined by $\hat{b}:=1/\hat{a}$.
Then $b \in \momspace{1}_2(\Z)$. Moreover, the following estimates hold:
\begin{align*}
&\mom{1}{2}(b) \leq 1/A^2 \mom{1}{2}(a),
\\
&\norm{b}_1 \leq \frac{1}{A} + \frac{\pi}{A^2\sqrt{3}} \mom{1}{2}(a).
\end{align*}
\end{theo}

\begin{coro}
\label{one_invariance}
The inclusion
\[
\momspace{\alpha}_2(\Z) \hookrightarrow B(\ell^2(\Z)),
\quad a \mapsto a*-
\]
is well-defined and preserves the spectrum of each element.
\end{coro}

\section{Proofs}
\label{proofs}
In order to prove the results announced in Section \ref{results}
we will need to introduce some facts about Sobolev algebras over the torus.
The following statements are adaptations to our needs of well-known facts.
\subsection{Chain rule and functional calculus}
\begin{lemma}
\label{chain}
Let $F: \R^2 \to \R$ be a $C^\infty$ compactly supported function.
Let $f:I^d \to \R^2$, $f=(f_1, f_2)$ be a measurable, locally integrable
function having a (locally integrable) weak derivative with respect to
$x_j$, ($1 \leq j \leq d$). Then $F \circ f$ has a weak derivative
with respect to $x_j$, given by
\[
(\deriv{x_1}F \circ f) \deriv{x_j}f_1 + (\deriv{x_2}F \circ f) \deriv{x_j}f_2.
\]
\end{lemma}
\begin{proof}
Let $g:= (\deriv{x_1}F \circ f) \deriv{x_j}f_1 + (\deriv{x_2}F \circ f)
\deriv{x_j}f_2$.
Observe that, since the derivatives of $F$ are bounded, $g$ is locally
integrable.

Given a smooth function $\varphi: I^d \to \R$ which has compact support $K$
inside the interior of $I^d$, we must show that
\[
\int_{I^d} (F \circ f) \deriv{x_j} \varphi = - \int_{I^d} g \varphi.
\]

There exists two sequences of smooth compactly supported functions
$\set{f^n_1}_n, \set{f^n_2}_n$, such that, for $i=1,2$,
$f^n_i \longrightarrow_n f_i$ and $\deriv{x_j}f^n_i \longrightarrow_n
\deriv{x_j}f_i$,
both in $L^1(K)$ and almost everywhere.
Let us call $f^n:=(f^1, f^2)$ and
$g^n:=(\deriv{x_1}F \circ f^n) \deriv{x_j}f^n_1 + (\deriv{x_2}F \circ f^n)
\deriv{x_j}f^n_2$.

Since all the functions involved are smooth, for each $n$, the following
equation
holds,
\[
\int_{I^d} (F \circ f^n) \deriv{x_j} \varphi = - \int_{I^d} g^n \varphi.
\]
Because $\varphi$ and $\deriv{x_j} \varphi$ are bounded and supported in $K$,
if we show that $F \circ f^n \longrightarrow_n F \circ f$
and $g^n \longrightarrow_n g$ in $L^1(K)$, the lemma will follow by letting
$n \longrightarrow \infty$.

Since $f^n_i \longrightarrow_n f_i$ in $L^1(K)$ and $F$ is Lipschitz,
it follows that $F \circ f^n \longrightarrow_n F \circ f$ in $L^1(K)$.

For $i=1,2$, let us estimate,
\begin{align*}
&\abs{(\deriv{x_i}F \circ f^n) \deriv{x_j} f^n_i - (\deriv{x_i}F \circ f)
\deriv{x_j} f_i}
\\
&\qquad \leq
\abs{\deriv{x_i}F \circ f^n - \deriv{x_i}F \circ f} \abs{\deriv{x_j}f_i}
+
\abs{\deriv{x_i}F \circ f^n} \abs{\deriv{x_j}f^n_i-\deriv{x_j}f_i}.
\end{align*}
The second term tend to $0$ in $L^1(K)$ because $\deriv{x_i}F$ is bounded.

For the first term, consider a measurable subset $E \subset K$, and
let us estimate,
\begin{align}
\nonumber
&\int_K \abs{\deriv{x_i}F \circ f^n - \deriv{x_i}F \circ f} \abs{\deriv{x_j}f_i}
\\
\nonumber
&\qquad= \int_E \abs{\deriv{x_i}F \circ f^n - \deriv{x_i}F \circ f} \abs{\deriv{x_j}f_i}
+ \int_{K \setminus E} \abs{\deriv{x_i}F \circ f^n - \deriv{x_i}F \circ f} \abs{\deriv{x_j}f_i}
\\
\label{chain_estimate}
&\qquad \leq
\sup_E \abs{\deriv{x_i}F \circ f^n - \deriv{x_i}F \circ f} \int_K
\abs{\deriv{f_i}}+2 \norm{\deriv{x_i}F}_\infty \int_{K \setminus E}
\abs{\deriv{x_j}f_i}.
\end{align}
Given $\varepsilon>0$, by the absolute continuity of the integral,
if $E$ is chosen such that the measure of $K \setminus E$ is sufficiently small,
the last term can be made smaller than $\varepsilon/2$ . Egorov's
Theorem grants the existence of such a set $E$ where, in addition,
$\deriv{x_i}F \circ f^n \longrightarrow_n \deriv{x_i}F \circ f$ uniformly.
Therefore, if $n$ if sufficiently big, the first term in \eqref{chain_estimate}
is less that $\varepsilon/2$. This completes the proof.
\end{proof}

We now prove a formula for the derivatives of the pointwise inverse
of a function. We would like to apply the previous lemma to the real and imaginary part
of the complex inversion $1/z$. Since they do not meet the hypothesis of the lemma,
we will use suitably localized versions of them.

\begin{coro}
\label{deriv_inv}
Let $f:I^d \to \C$, be a measurable function satisfying
$0 < A \leq \abs{f} \leq B < \infty$, for some constants $A,B$.
Suppose that $f$ has a (locally integrable) weak derivative with respect to
$x_j$, ($1 \leq j \leq d$). Then $1/f$ has a weak derivative
with respect to $x_j$, given by
\[
\deriv{x_j}(1/f) = -\deriv{x_j}(f) / f^2.
\]
\end{coro}
\begin{proof}
Let us set $f = f_1+f_2i$, where $f_1, f_2:I^d \to \R$.

Let $G=(G_1,G_2): \R^2 \setminus \set{0} \to \R^2$ be defined by $G(x,y):= (x^2
+ y^2)^{-1} (x, -y)$, so under the identification $\R^2 = \C$, $G$ is the complex inversion.
Let $\eta: \R^2 \to \R$ be a smooth compactly supported function such that $\eta \equiv 1$ on a
neighborhood of the crown $C:=\set{(x,y) \in \R^2: A \leq \norm{(x,y)}_2 \leq B}$, not containing 0.

The function $F=(F_1,F_2)$ defined by $F(x,y):=G(\eta(x,y)x,\eta(x,y)y)$ is
smooth, compactly supported and agrees with $G$ on a neighborhood of the crown $C$.
Observe that on $C$, the partial derivatives of $F$ agree with those of $G$ as well.

Consequently $F \circ f= G \circ f = 1/f=u+vi$, where $u,v:I^d \to \R$.
We can apply Lemma \ref{chain}, to find that $u$ and $v$ have weak derivatives
with respect to $x_j$ given by
\begin{align*}
\deriv{x_j}u &= (\deriv{x_1}F_1 \circ f) \deriv{x_j}f_1 + (\deriv{x_2}F_1 \circ f) \deriv{x_j}f_2
= (\deriv{x_1}G_1 \circ f) \deriv{x_j}f_1 + (\deriv{x_2}G_1 \circ f) \deriv{x_j}f_2,
\\
\deriv{x_j}v &= (\deriv{x_1}F_2 \circ f) \deriv{x_j}f_1 + (\deriv{x_2}F_2 \circ f) \deriv{x_j}f_2
= (\deriv{x_1}G_2 \circ f) \deriv{x_j}f_1 + (\deriv{x_2}G_2 \circ f) \deriv{x_j}f_2.
\end{align*}
Since $G$ is a holomorphic function, Cauchy-Riemann equations say that
$\deriv{x_2}(G^1)=-\deriv{x_1}(G^2)$ and $\deriv{x_1}(G^1)=\deriv{x_2}(G^2)$.
This yields,
\begin{align*}
\deriv{x_j}(1/f) &= \deriv{x_j}u+\deriv{x_j}vi
\\
&=
(\deriv{x_1}G_1 \circ f)(\deriv{x_j}f_1+\deriv{x_j}f_2i)
+
(\deriv{x_1}G_2 \circ f)i(\deriv{x_j}f_1+\deriv{x_j}f_2i)
\\
&= (\deriv{x_1}G \circ f) \deriv{x_j}f.
\end{align*}
A direct computation now shows now that $\deriv{x_j}(1/f)= -\deriv{x_j}f/f^2$.
\end{proof}

\subsection{Some facts about Sobolev algebras}
For $1 \leq j \leq d$ and $n \geq 0$, let us define the $j$-directional
Sobolev algebra by,
\[
\sob{j,n}=\sob{j,n}(I^d) := \set{f:I^d \to \C/ \de{ke_j}(f) \in L^\infty(I^d),
\forall 0 \leq k \leq n},
\]
where the derivatives are distributional.
For a multi-index $\alpha \in \N_0^d$, let us define
\[
\sob{\alpha}=\sob{\alpha}(I^d) := \bigcap_{j=1}^d \sob{j,\alpha_j}(I^d).
\]
Observe that $\momspace{\alpha}_{op}=\mathcal{F}^{-1}(\sob{\alpha})$.
\begin{lemma}
\label{product_rule}
\mbox{}

\begin{itemize}
\item[(a)] If for some index $j$, $f,g \in \sob{j,1}$, then $fg \in \sob{j,1}$
and
$\deriv{x_j}(fg)=\deriv{x_j}(f)g+f\deriv{x_j}(g)$.
\item[(b)] If for some indexes $j,n$, $f,g \in \sob{j,n}$, then $fg \in \sob{j,n}$.
Moreover, the following formula holds.
\[
\de{ne_j}(fg)=\sum_{h=0}^n \binom{n}{h} \de{he_j}f \de{(n-h)e_j}g.
\]
\item[(c)] If for some multi-index $\alpha$, $f,g \in \sob{\alpha}$, then $fg \in \sob{\alpha}$.
Moreover, the following formula holds.
\[
\de{\alpha}(fg)=\sum_{\beta \leq \alpha} \binom{\alpha}{\beta} \de{\beta}f \de{\alpha-\beta}g.
\]
\end{itemize}
\end{lemma}
\begin{rem}
\label{conv_alg}
This lemma implies that $\momspace{\alpha}_{op}=\mathcal{F}^{-1}(\sob{\alpha})$
is a convolution algebra.
\end{rem}
\begin{proof}
(a) This is a slight variation of \cite[pg 129, Theorem 4 (i)]{evga92},
the proof given there also applies to this case.

(b) To prove that $fg \in \sob{j,n}$, we proceed by induction on $n$.
The case $n=0$ is clear. Suppose that
$f,g \in \sob{j,n}$, for some $n \geq 1$. Then by (a),
$\deriv{x_j}(fg)=\deriv{x_j}(f)g+f\deriv{x_j}(g)$.
Since $\deriv{x_j}(f),g,f,\deriv{x_j}(g)$ all belong to $\sob{j,n-1}$, by the
inductive
hypothesis $\deriv{x_j}(fg) \in \sob{j,n-1}$ as well. Hence, $fg \in \sob{j,n}$.

The prescribed formula for the derivatives of $fg$ follow from (a) by induction.

(c) is an consequence of item (b); the formula can be proved similarly to that in (b).
\end{proof}

\begin{lemma}
\label{sob_inv}
Suppose that $f:I^d \to \C$ is such that $0 < A \leq \abs{f} \leq B < \infty$ holds
for some constants $A,B$.
\begin{itemize}
\item[(a)] If for some indexes $j,n$, $f \in \sob{j,n}$, then $1/f \in \sob{j,n}$.
\item[(b)] If for some multi-index $\alpha$, $f \in \sob{\alpha}$, then $1/f \in \sob{\alpha}$.
\end{itemize}
\end{lemma}
\begin{proof}
For (a), we proceed by induction on $n$. If $n=0$, the conclusion is evident. Suppose that
$f \in \sob{n+1,j}$, for some $n \geq 0$. Then, by Corollary \ref{deriv_inv},
$1/f$ has a weak derivative with respect to $x_j$
given by $\deriv{x_j}(1/f)=-\deriv{x_j}f f^{-2}$.

Since $0 < A^2 \leq \abs{f^2} \leq B^2 < \infty$ holds, and by Lemma \ref{product_rule},
$f^2 \in \sob{n+1,j} \subseteq \sob{n,j}$, the inductive hypothesis implies that
$1/f^2 \in \sob{n,j}$. In addition, since $f \in \sob{n+1,j}$, 
it follows  that $\deriv{x_j}f \in \sob{n,j}$, so Lemma \ref{product_rule}
implies that $\deriv{x_j}(1/f) \in \sob{n,j}$. This shows that $1/f \in \sob{n+1,j}$.

(b) is a direct consequence of (a).
\end{proof}

\subsection{Proofs}
Now we can prove the results announced in Section \ref{results}.
\begin{proof}[Proof of Theorem \ref{multi_theo}]
Let $f:=\hat{a}$ and $g:=\hat{b}=1/f$.
Since $a \in \momspace{\alpha}_{op}$, $f \in \sob{\alpha}$. By Lemma \ref{sob_inv}, we know that
$g \in \sob{\alpha}$, so $b \in \momspace{\alpha}_{op}$.

First observe that $\mom{0}{op}(b) = \norm{g}_\infty \leq 1/A$.
Let $0 \not= \gamma \leq \alpha$. By \ref{product_rule} we have,
\[
0 = \de{\gamma}(gf)
= \sum_{\beta \leq \gamma} \binom{\gamma}{\beta} \de{\beta}g \de{\gamma-\beta}f.
\]
Hence,
\[
\de{\gamma}g = -1/f \sum_{\beta < \gamma} \binom{\gamma}{\beta} \de{\beta}g \de{\gamma-\beta}f.
\]
This yields,
\[
\norm{\de{\gamma}g}_\infty \leq
1/A \sum_{\beta < \gamma} \binom{\gamma}{\beta} \norm{\de{\beta}g}_\infty \norm{\de{\gamma-\beta}f}_\infty.
\]
Consequently,
\[
(2 \pi)^{\abs{\gamma}} \mom{\gamma}{op}(b)
\leq
1/A \sum_{\beta < \gamma} \binom{\gamma}{\beta} (2 \pi)^{\abs{\beta}+\abs{\gamma-\beta}} \mom{\beta}{op}(b) \mom{\gamma-\beta}{op}(a).
\]
and the desired formula follows.
\end{proof}

\begin{proof}[Proof of Corollary \ref{multi_invariance}]
Suppose that $a \in \momspace{\alpha}_{op}(\Z^d)$ is such that
the convolution operator $a*\mbox{-}: \ltwo \to \ltwo$ is invertible.

The operator $\f (a*\mbox{-}) \f^{-1}: L^2 \to L^2$ is then invertible.
Since this is a multiplication operator with symbol $\hat{a}$, it
follows that $\abs{\hat{a}}$ must be bounded below.
Now Theorem \ref{multi_theo} implies the existence of
$b \in \momspace{\alpha}_{op}(\Z^d)$ such that $a*b = \delta$.
So $a$ is invertible in $\momspace{\alpha}_{op}(\Z^d)$.
\end{proof}

\begin{proof}[Proof of Theorem \ref{one_theo}]
Let us set $f:= \hat{a}$. Since $a \in \momspace{2}_2$, $f \in L^2(I)$
and $f$ has a weak derivative $f' \in L^2(I)$. Moreover, by Corollary
\ref{deriv_inv},
\[
(1/f)'=-f'/f^2.
\]
Consequently,
\begin{align*}
\mom{1}{2}(b) &= \norm{\momd{1}(b)}_2 = \norm{\f(\momd{1}(b))}_2
\\
&= 1/(2\pi) \norm{(1/f)'}_2 \leq 1/(2\pi) 1/A^2 \norm{f'}_2
\\
&= 1/A^2 \norm{\f(\momd{1}(a))}_2 = 1/A^2 \norm{\momd{1}(a)}_2
\\
&= 1/A^2 \mom{1}{2}(a).
\end{align*}

To prove the estimate on $\norm{b}_1$, first observe that
$\norm{b}_\infty \leq \norm{1/f}_{L^1} \leq 1/A$.
Finally let us estimate,
\begin{align*}
\sum_k \abs{b_k} &= \abs{b_0} + \sum_{k \not= 0} \abs{b_k}
\leq \frac{1}{A} + \left(\sum_{k \not= 0} 1/k^2 \right)^{1/2} \mom{1}{2}(b)
\\
&\leq \frac{1}{A} + \frac{\pi}{A^2\sqrt{3}} \mom{1}{2}(a).
\end{align*}

\end{proof}

\begin{proof}[Proof of Corollary \ref{one_invariance}]
First observe that if $a \in \momspace{1}_2(\Z)$, then
$\hat{a}$ has a weak derivative in $L^2(I)$.
Consequently, $\hat{a}$ is Holder continuous (with exponent 1/2)
and is therefore bounded on $I$. This shows that $a*-$ is a well-defined,
bounded operator on $\ell^2(\Z)$.

Suppose that $a \in \momspace{1}_2(\Z)$ is such that
the convolution operator $a*\mbox{-}: \ltwo \to \ltwo$ is invertible.
As noted in the proof of Corollary \ref{multi_invariance},
it follows that there exists constants $0<A \leq B < \infty$,
such that $A \leq \abs{\hat{a}} \leq B$. Now Theorem \ref{one_theo}
implies that $a$ is invertible in $\momspace{1}_2(\Z)$.
\end{proof}

\section{Applications}
\label{apps}
We now show how the estimates of the previous sections can be used to
give explicit localization bounds for spline-type spaces. We will work
in dimension 1 and use Theorem \ref{one_theo}.
Similar applications can be done for arbitrary dimension by means
of the iterative estimate given in Theorem \ref{multi_theo}. However
for clarity we restrict ourselves to dimension 1.

Given a measurable function $\varphi:\R \to \C$ and $1 \leq p \leq \infty$
we consider $S^p$, the $L^p$-closed linear span of the set translates
$\set{\varphi(\cdot-k):k\in\Z}$. We call $S^p$ a (principal) spline-type space
if this set of translates forms a $p$-Riesz sequence (see definition below)
and the generator function $\varphi$ satisfies some decay conditions described below.

When $p=2$, the hypothesis of the set $\set{\varphi(\cdot-k):k\in\Z}$
being a Riesz sequence can be studied by means of the \emph{gramian function}
$\sum_k \abs{\hat{\varphi}(\cdot-k)}^2$. As said in the introduction
it has been shown \cite{algr01} that inverse-closedness principles like Wiener's $1/f$
lemma can be used to extend the Riesz sequence conditions to the general values of $p$.
We will apply the estimates of Section \ref{results} to this end and then
obtain some qualitative consequences for the bounds.

Let us first introduce some definitions and notation to be used in this section.

We call a sequence of measurable functions $\set{f_k}_k$ a $p$-\emph{Riesz sequence}
if it is a $p$-Riesz basis of its $L^p$-closed linear span, that is, if
there exists constants $A_p, B_p >0$ such that the norm equivalence
\[
A_p \norm{c}_{\ell^p} \leq \norm{\sum_k c_k f_k}_{L^p} \leq B_p \norm{c}_{\ell^p}
\]
holds for all finitely supported sequences of complex numbers $c$. In this case
it follows that the series $\sum_k c_k f_k$ converges unconditionally in $L^p$
for any sequence $c \in \ell^p$ and the norm equivalence extends to these sequences.

When $p=2$, the Riesz sequence condition implies the existence of a bi-orthogonal
system $\set{g_k}_k$ (i.e. $\ip{f_k}{g_j}=\delta_{k,j}$) belonging to the
closed linear span of $\set{f_k}_k$. When the functions $f_k$ have the special structure
$f_k=f(\cdot-k)$, $k \in \Z$, the bi-orthogonal system has the same structure $g_k=g(\cdot-k)$
and we call $g$ the window dual to $f$.

We will consider the following constants,
\begin{align*}
K_\alpha &:= \sup_{x \in \R} (1+\abs{x})^\alpha
		\int_R (1+\abs{s})^{-\alpha} (1+\abs{x-s})^{-\alpha} ds
\mbox{, for $\alpha >1$,}
\\
W_\alpha &:= 2 \sum_{j \geq 1} j^{-\alpha}
\mbox{, for $\alpha >1$,}
\\
S_\alpha &:= \left( \sum_{k \in \Z} k^2 (1+\abs{k})^{-2\alpha} \right)^{1/2}
\mbox{, for $\alpha >3/2$.}
\end{align*}
The constant $S_\alpha$ is easily seen to be finite by comparing it to a $p$-harmonic
series with $p=2(\alpha-1)$.
$K_\alpha$ can be bounded by $(1+2^\alpha) \int_\R (1+\abs{s})^{-\alpha} ds$.

Let $\delta>0$ and $X \equiv \set{x_k}_{k \in \Z} \subset \R$ be a
countable indexed set of points. We call the number 
\[
N(X):=\sup_{k \in \Z} \#\set{j \in \Z: x_j \in [k,k+1)}
\]
the \emph{relative separation} of $X$ and we will say that $X$ is relatively
separated if $N(X) < \infty$.
We will say that the set $X$ is $\delta$-dense if the family of open intervals
$\set{(x_k-\delta, x_k+\delta)}_{k\in\Z}$ covers $\R$.

Given $\delta >0$ and a function $f:\R \to \C$, we define
its \emph{$\delta$-oscillation} by
\[
\osc_\delta(f)(x) := \sup_{\abs{x-y}\leq \delta} \abs{f(x)-f(y)}.
\]

For a measurable function $f:\R \to \C$ and $1 \leq p,q \leq \infty$
we consider the norm
\[
\norm{f}_{W(L^p,\ell^q)} := \norm{ \left( \norm{f}_{L^p([0,1]+j)} \right)_{j\in\Z} }_{\ell^q},
\]
and denote by $W(L^p,\ell^q)$ the space of all measurable functions with finite 
$\norm{\cdot}_{W(L^p,\ell^q)}$ norm.

The space $W(L^p,\ell^q)$ is an example of a wide class of spaces known as \emph{amalgam spaces}
or spaces of Wiener-type (see \cite{fe83}). We will require some well-know facts
that are special cases of the general multiplier theory developed in \cite{fe83}.
\begin{fact}
\label{amalgam_facts}
Let $1 \leq p \leq \infty$ and $g \in W(L^\infty, \ell^p)$.
\begin{itemize}
\item[(a)] If $f \in L^p(\R)$ and $c_k := \ip{f}{g(\cdot-k)}$, for $k \in \Z$,
then $\norm{c}_{\ell^p} \leq \norm{f}_{L^p} \norm{g}_{W(L^\infty,\ell^1)}$.
\item[(b)] If $c \in \ell^p(\Z)$ and $f = \sum_k c_k g(\cdot-k)$, then
$\norm{f}_{W(L^\infty,\ell^p)} \leq \norm{c}_{\ell^p} \norm{g}_{W(L^\infty,\ell^1)}$.
\item[(c)] If $X \equiv \set{x_k}_k$ is a relatively separated set of points,
then
\[
\norm{(f(x_k))_k}_{\ell^p} \leq N(X)^{1/p} \norm{f}_{W(L^\infty,\ell^p)}.
\]
\end{itemize}
\end{fact}

\subsection{Localization in spline-type spaces}
We now show that for a well-localized generator function $\varphi$,
if the set of integer translates $\set{\varphi(\cdot-k):k \in \Z}$ forms
an $L^2$ Riesz sequence with lower bound $A$, then they form
an $L^p$ Riesz sequence for all $1\leq p \leq \infty$,
with an uniform lower bound $\approx \frac{A^2}{A+1}$.
Moreover, all the constants involved are explicit. In order to do this
it suffices to give an estimate on the decay of the dual window $\psi$ (see \cite{algr01}.)
\begin{theo}
\label{app_loc}
Let $\varphi: \R \to \C$ be a measurable function such that
\[
\abs{\varphi(x)} \leq C(1+\abs{x})^{-\alpha}
\mbox{, for $x \in \R$,}
\]
holds for some constants $C>0$ and $\alpha>3/2$.
Suppose that,
\[
0 < A \leq \sum_k \abs{\hat{\varphi}(\cdot-k)}^2 \leq B < \infty,
\]
holds for some constants $A,B$ and let $\psi$ be the window dual to $\varphi$.
Then, both $\varphi$ and $\psi$ belong to $W(L^\infty,\ell^1)$ and the following
norm estimates hold:
\begin{align*}
&\norm{\varphi}_{W(L^\infty,\ell^1)} \leq C W_\alpha,
\\
&\norm{\psi}_{W(L^\infty,\ell^1)} \leq
\left(\frac{1}{A} + \frac{\pi}{A^2\sqrt{3}} C^2 K_\alpha S_\alpha \right) C W_\alpha.
\end{align*}
\end{theo}
\begin{coro}
\label{app_loc_coro}
Let $\varphi: \R \to \C$ be a measurable function such that
\[
\abs{\varphi(x)} \leq C(1+\abs{x})^{-\alpha}
\mbox{, for $x \in \R$,}
\]
holds for some constants $C>0$ and $\alpha>3/2$ and that
\[
0 < A \leq \sum_k \abs{\hat{\varphi}(\cdot-k)}^2 \leq B < \infty,
\]
holds for some constants $A,B$.

Then, for every $1 \leq p \leq \infty$ the following relation holds,
\begin{align*}
r \norm{c}_{\ell^p} &\leq \norm{\sum_k c_k \varphi(\cdot-k)}_{L^p} \leq R \norm{c}_{\ell^p}
\mbox{, for $c \in \ell^p$,}
\\
\mbox{where, }
r &= C^{-1} W_\alpha^{-1}
\frac{A^2}{A+\frac{\pi}{\sqrt{3}} C^2 K_\alpha S_\alpha},
\\
R &= C W_\alpha.
\end{align*}
\end{coro}
\begin{rem}
\label{qrem}
Observe that for $r \approx \frac{A^2}{A+1}$. Moreover, the implicit constants are continuously
determined by $C$ and $\alpha$, the parameters measuring the decay of $\varphi$.
\end{rem}
\begin{proof}[Proof of Theorem \ref{app_loc}]
A direct computation shows that
$\norm{\varphi}_{W(L^\infty, \ell^1)} \leq C W_\alpha$.

For $k \in \Z$, let us set $a_k:=\ip{\varphi}{\varphi(\cdot+k)}$.
Then $\abs{a_k} \leq C^2 K_\alpha (1+\abs{k})^{-\alpha}$
and $\mom{1}{2}(a) \leq C^2 K_\alpha S_\alpha$.

Poisson's summation formula implies that $\hat{a}=\sum_k \abs{\hat{\varphi}(\cdot-k)}^2$. Therefore,
Theorem \ref{one_theo} implies that there exists $b \in \momspace{1}_2(\ell^2)$ such that
$a*b=\delta$. Moreover $\norm{b}_1 \leq \frac{1}{A} + \frac{\pi}{A^2\sqrt{3}} \mom{1}{2}(a)$.

If we let $\psi := \sum_k b_k \varphi(\cdot+k)$, it follows that $\psi \in S^2$ is
the window dual to $\varphi$ and we simply estimate,
\begin{align*}
\norm{\psi}_{W(L^\infty, \ell^1)} &\leq  \norm{b}_1 \norm{\varphi}_{W(L^\infty, \ell^1)}
\\
&\leq \left(\frac{1}{A} + \frac{\pi}{A^2\sqrt{3}} C^2 K_\alpha S_\alpha \right)
     \norm{\varphi}_{W(L^\infty, \ell^1)}.
\end{align*}
The conclusion follows by
combining this estimate with the bound given for $\norm{\varphi}_{W(L^\infty, \ell^1)}$.
\end{proof}

\begin{proof}[Proof of Corollary \ref{app_loc_coro}]
Let $1 \leq p \leq \infty$ and $c$ be a finitely supported sequence of complex numbers.
Set $f := \sum_k c_k \varphi(\cdot-k)$.

According to Theorem \ref{app_loc},
\begin{align*}
\norm{f}_{L^p} &\leq \norm{f}_{W(L^\infty,\ell^p)} \leq \norm{c}_{\ell^p} \norm{\varphi}_{W(L^\infty,\ell^1)}
\\
&\leq C W_\alpha \norm{c}_{\ell^p}.
\end{align*}
Let us show the lower bound holds. With the notation of Theorem \ref{app_loc},
observe that $c_k = \ip{f}{\psi(\cdot-k)}$, for all $k$. Hence,
\[
\norm{c}_{\ell^p} \leq \norm{f}_{L^p} \norm{\psi}_{W(L^\infty,\ell^1)}.
\]
Now the estimate in Theorem \ref{app_loc} yields the lower bound. An approximation
argument proves the norm equivalence for arbitrary $c \in \ell^p$.
\end{proof}
\subsection{Oversampling from oscillation estimates}
Now we illustrate how oversampling bounds can be obtained from the estimates of the
previous sections.
\begin{theo}
\label{over}
Let $\varphi: \R \to \C$ be a function such that
\[
\abs{\varphi(x)} \leq C(1+\abs{x})^{-\alpha}
\mbox{, for $x \in \R$,}
\]
holds for some constants $C>0$ and $\alpha>3/2$.
Suppose that,
\[
0 < A \leq \sum_k \abs{\hat{\varphi}(\cdot-k)}^2 \leq B < \infty,
\]
holds for some constants $A,B$, and that $\varphi$ has a weak
derivative $\varphi' \in W(L^q, \ell^1)$ for some $1<q\leq\infty$.

Let $\delta>0$ be such that
\[
\rho := C^3 W_\alpha^3
\left(\frac{1}{A} + \frac{\pi}{A^2\sqrt{3}} C^2 K_\alpha S_\alpha \right)^2
(2 \ceil{\delta}+1) \norm{\varphi'}_{W(L^q,\ell^1)} \delta^{1-1/q} < 1.
\]
Then, for every relatively separated set $X \equiv \set{x_k}_{k\in\Z}$
that is $\delta$-dense and any $1 \leq p \leq \infty$, the following sampling
estimate holds,
\[
c_p \norm{f}_{L^p} \leq \norm{(f(x_k))_{k\in \Z}}_{\ell^p} \leq C_p \norm{f}_{L^p},
\qquad \mbox{for all $f \in S^p$},
\]
\begin{align*}
\mbox{where }
C_p &= N(X)^{1/p} C^2 W_\alpha^2 \left(\frac{1}{A} + \frac{\pi}{A^2\sqrt{3}} C^2 K_\alpha S_\alpha \right),
\\
c_p &= (1-\rho)(2\delta)^{-1/p} C^{-2} W_\alpha^{-2}
\frac{A^2}{A+\frac{\pi}{\sqrt{3}} C^2 K_\alpha S_\alpha},
\end{align*}
and $S^p$ stands for the $L^p$-closed linear span of $\set{\varphi(\cdot-k):k\in\Z}$.
\end{theo}
\begin{rem}
Since the above result is just an application of the general sampling scheme
developed in \cite{alfe98}, it follows that the effective reconstruction
procedure developed there also applies here.
\end{rem}

In order to prove the theorem, we will need this slight adaptation of one of the
results in \cite{alfe98} (see also \cite{algr01}).
\begin{prop}
\label{alfe_fork}
Let $\varphi \in W(L^\infty,\ell^1)$ be such that the set of integer translates
$\set{\varphi(\cdot-k):k\in\Z}$ is an $p$-Riesz basis of its $L^p$-closed
linear span $S^p$.
Suppose that there exists $\psi \in W(L^\infty,\ell^1)$ such that the sets
$\set{\varphi(\cdot-k):k\in\Z}$ and $\set{\psi(\cdot-k):k\in\Z}$
are bi-orthogonal.

Let $\delta>0$ be such that
\[
\gamma:=\norm{\varphi}_{W(L^\infty.\ell^1)} \norm{\psi}_{W(L^\infty.\ell^1)}^2
\norm{\osc_\delta(\varphi)}_{W(L^\infty.\ell^1)}<1.
\]
Then, for every relatively separated set $X \equiv \set{x_k}_{k\in\Z}$
that is $\delta$-dense the following sampling estimate holds,
\[
c_p \norm{f}_{L^p} \leq \norm{(f(x_k))_{k\in \Z}}_{\ell^p} \leq C_p \norm{f}_{L^p},
\qquad \mbox{for all $f \in S^p(\varphi)$},
\]
\begin{align*}
\mbox{where } C_p &= N(X)^{1/p} \norm{\varphi}_{W(L^\infty,\ell^1)} \norm{\psi}_{W(L^\infty,\ell^1)},
\\
c_p &= (1-\gamma)(2\delta)^{-1/p}
\norm{\varphi}_{W(L^\infty,\ell^1)}^{-1} \norm{\psi}_{W(L^\infty,\ell^1)}^{-1},
\end{align*}
where $1/p=0$, if $p=\infty$.
\end{prop}
\begin{proof}
Consider the operator $P:L^p(\R) \to L^p(\R)$ given by
\[P(f):=\sum_{k\in\Z^d} \ip{f}{\psi(\cdot-k)} \varphi(\cdot-k).\]
Since $\psi \in W(C_0,\ell^1)$ it follows that $P$ is well defined
and that
\[
\norm{P} \leq \norm{\varphi}_{W(L^\infty,\ell^1)} \norm{\psi}_{W(L^\infty,\ell^1)}.
\]
Moreover, since the sets $\set{\varphi(\cdot-k):k\in\Z}$ and
$\set{\psi(\cdot-k):k\in\Z}$ are bi-orthogonal it follows
that $P$ is a projector on $S^p$, the $L^p$-closed linear span of $\set{\varphi(\cdot-k):k\in\Z}$.

Since $X$ is $\delta$-dense, the family $\set{(x_k-\delta,x_k+\delta): k\in\Z}$
covers $\R^d$. Let $\set{g_k: k \in \Z}$ be a nonnegative partition
of the unity subordinated to it.

Consider the operators,
\begin{align*}
&Z: S^p \to \ell^p,
\qquad
Z(f):=(f(x_k))_{k\in\Z},
\\
&I: \ell^p \to L^p,
\qquad
I(c) := \sum_{k\in\Z} c_k g_k,
\qquad
\\
&PIZ: S^p(\varphi) \to S^p(\varphi).
\end{align*}

The argument in \cite{alfe98} shows that $\norm{PIZ-Id}_{B(S^p)} \leq \gamma < 1$.

Hence, the operator $Z$ is left-invertible and $X$ is a $p$-sampling set
for $S^p$. Let us establish the announced bounds. For $f\in S^p(\varphi)$,
\[
\norm{Z(f)}_{\ell^p} \leq N(X)^{1/p} \norm{f}_{W(L^\infty,\ell^p)}
\leq N(X)^{1/p} \norm{\varphi}_{W(L^\infty,\ell^1)} \norm{\psi}_{W(L^\infty,\ell^1)}
\norm{f}_{L^p}.
\]
This establishes the upper bounds, for the lower one, let us first estimate
the operator norm of $I$.
For $c \in \ell^p$ and $p<\infty$, using H\"older's inequality,
\begin{align*}
\norm{I(c)}_{\ell^p}^p &\leq \int_\R \left( \sum_k \abs{c_k} g_k(x) \right)^p dx
\\
&\leq \int_\R \sum_k \abs{c_k}^p g_k(x) \left( \sum_k g_k(x) \right)^{p/p'} dx
\\
&= \sum_k \abs{c_k}^p \int_\R g_k(x) dx \leq 2\delta \sum_k \abs{c_k}^p 
\end{align*}
So $\norm{I} \leq (2\delta)^{1/p}$. The same estimate is easily established
when $p=\infty$ if we interpret $1/p$ as $0$.
Now for $f \in S^p$,
\[
\norm{f}_{L^p} \leq \norm{f-PIZ(f)}_{L^p} + \norm{PIZ(f)}_{L^p}
\leq \gamma \norm{f}_{L^p} + \norm{P} \norm{I} \norm{Z(f)}_{\ell^p}.
\]
So,
\begin{align*}
\norm{f}_{L^p}
&\leq \frac{1}{1-\gamma} \norm{P} \norm{I} \norm{Z(f)}_{\ell^p}
\\
&\leq \frac{1}{1-\gamma}
\norm{\varphi}_{W(L^\infty,\ell^1)} \norm{\psi}_{W(L^\infty,\ell^1)}
(2\delta)^{1/p} \norm{Z(f)}_{\ell^p},
\end{align*}
and the estimate given follows.
\end{proof}
Now we can combine the last proposition with the estimates of the previous sections
to prove Theorem \ref{over}
\begin{proof}[Proof of Theorem \ref{over}]
Let $\gamma$ be defined as in the statement of Proposition \ref{alfe_fork}.

For any $x,y \in \R$ such that $\abs{x-y}\leq\delta$,
if we call $J$ the interval determined by $x$ and $y$
we can estimate,
\[
\abs{\varphi(x)-\varphi(y)} \leq \int_J \abs{\varphi'}
\leq \norm{\varphi'}_{L^q(J)} \delta^{1-1/q}.
\]
Therefore, for every $x\in\R$, we have that
$\osc_\delta{\varphi}(x) \leq \norm{\varphi'}_{L^q([x-\delta,x+\delta])} \delta^{1-1/q}$
and for $j \in \Z$,
$\sup_{[0,1]+j} \osc_\delta(\varphi) \leq \norm{\varphi'}_{L^q([j-\delta,j+1+\delta])} \delta^{1-1/q}$.

For any integer $j$ the inclusion
\[
[j-\delta,j+1+\delta] \subseteq \bigcup_{h=-\ceil{\delta}}^{\ceil{\delta}} [j+h,j+h+1],
\]
and Minkowski's inequality implies that
\[
\norm{\varphi'}_{L^q([j-\delta,j+1+\delta])} \leq \sum_{h=-\ceil{\delta}}^{\ceil{\delta}}
\norm{\varphi'}_{L^q([j+h,j+h+1])}.
\]
Consequently, 
\begin{align*}
\norm{\osc_\delta(\varphi)}_{W(L^\infty,\ell^1)} &\leq
\sum_{j\in\Z} \sum_{h=-\ceil{\delta}}^{\ceil{\delta}} \norm{\varphi'}_{L^q([j+h,j+h+1])}
\delta^{1-1/q}
\\
&= \sum_{h=-\ceil{\delta}}^{\ceil{\delta}} \sum_{j\in\Z} \norm{\varphi'}_{L^q([j+h,j+h+1])}
\\
&= (2 \ceil{\delta}+1) \norm{\varphi'}_{W(L^q,\ell^1)} \delta^{1-1/q}.
\end{align*}
We combine this with the estimates of Theorem \ref{app_loc} 
to conclude that $\gamma \leq \rho <1$.

It follows now from Proposition \ref{alfe_fork}
that every relatively separated, $\delta$-dense set is a $p$-sampling set for $S^p$
for all $1 \leq p \leq \infty$, with the announced bounds.
\end{proof}
\section{Acknowledgements}
The author holds a fellowship from the CONICET and thanks this institution for its support.
His research is also partially supported by Grants: PICT 15033, CONICET, PIP 5650, UBACyT  X108.


\begin{thebibliography}{30}
\bibitem{alfe98}
A.~{A}ldroubi and H.~G. {F}eichtinger.
\newblock {E}xact iterative reconstruction algorithm for multivariate
irregularly sampled functions in spline-like spaces: {T}he ${L}^p$-{T}heory.
\newblock {\em {P}roc. {A}mer. {M}ath. {S}oc.}, {\bf 126}(9)  (1981)  2677--2686.

\bibitem{algr01}
A.~{A}ldroubi and K.~{G}r{\"o}chenig.
\newblock {N}onuniform sampling and reconstruction in shift-invariant
spaces.
\newblock {\em {S}{I}{A}{M} {R}ev.}, {\bf43}(4) (2001) 585--620.

\bibitem{bacahela06}
R.~Balan, P.G. Casazza, C.~Heil, and Z.~Landau.
\newblock {Density, Overcompleteness, and Localization of Frames. I. Theory}.
\newblock {\em Journal of Fourier Analysis and Applications}, 12(2):105--143,
  2006.

\bibitem{evga92}
L.~Evans and R.~Gariepy.
\newblock {M}easure Theory and Fine Properties of Functions.
\newblock CRC Press, 1992.

\bibitem{fe83}
H.~G.~{F}eichtinger.
\newblock {B}anach convolution algebras of {{W}}iener type.
\newblock In {\em {P}roc. {C}onf. on {F}unctions, {S}eries, {O}perators,
  {B}udapest 1980}, volume~35 of {\em {C}olloq. {M}ath. {S}oc. {J}anos
  {B}olyai}. {N}orth-{H}olland, {A}msterdam (1983) 509--524.

\bibitem{gr04}
K.~Gr{\"o}chenig.
\newblock {Localization of Frames, Banach Frames, and the Invertibility of the
  Frame Operator}.
\newblock {\em Journal of Fourier Analysis and Applications}, 10(2):105--132,
  2004.

\end{thebibliography}
\def\cprime{$'$}

\end{document}